\documentclass[12pt]{article}
\usepackage{amsmath,amsthm}
\usepackage{amsfonts,amsmath}
\usepackage{amssymb,hyperref}
\usepackage{fullpage}
\usepackage{epsf}
\usepackage{color}
\usepackage{graphicx}
\usepackage[affil-it]{authblk}

\usepackage{mathrsfs}
\usepackage{pgf,tikz}
\usepackage{multirow}

\newtheorem{theorem}{Theorem}[section]
\newtheorem{lemma}[theorem]{Lemma}
\newtheorem{corollary}[theorem]{Corollary}

\newtheorem{example}[theorem]{Example}

\renewcommand{\geq}{\geqslant}
\renewcommand{\leq}{\leqslant}

\renewcommand{\le}{\leqslant}
\renewcommand\emptyset{\varnothing}
\newcommand{\MOFS}{{\rm MOFS}}
\newcommand{\DK}{{\ensuremath{\rm DK}}}
\newcommand{\PBD}{{\ensuremath{\rm PBD}}}

\def \cB {{\mathcal B}}

\def \cD {{\mathcal D}}

\def \cF {{\mathcal F}}

\def \cK {{\mathcal K}}

\def \cP {{\mathcal P}}

\def \cV {{\mathcal V}}

\def\Z{\mathbb{Z}}

\def\eref#1{$(\ref{#1})$}
\def\egref#1{Example~$\ref{#1}$}

\def\lref#1{Lemma~$\ref{#1}$}
\def\tref#1{Theorem~$\ref{#1}$}

\def\cref#1{Conjecture~$\ref{#1}$}
\def\Cref#1{Corollary~$\ref{#1}$}

\title{Maximal sets of mutually orthogonal frequency squares and Doehlert-Klee designs}

\author{Carly Bodkin\thanks{School of Mathematics, Monash University, Clayton Vic 3800, Australia.
\texttt{carly.bodkin@monash.edu}}}
\author{Nicholas J.~Cavenagh\thanks{Department of Mathematics, The University of Waikato, Private Bag 3105, Hamilton 3240, New Zealand.
\texttt{nickc@waikato.ac.nz}}}
\author{Ian M.~Wanless\thanks{School of Mathematics, Monash University, Clayton Vic 3800, Australia.
\texttt{ian.wanless@monash.edu}}}
\affil{}

\date{}

\begin{document}

\maketitle

\begin{abstract}
  A \emph{binary frequency square} of type $(n;\lambda_0,\lambda_1)$
  is a $(0,1)$-matrix of order $n$ with $\lambda_0$ zeros and
  $\lambda_1$ ones in each row and in each column. Two such squares
  are \emph{orthogonal} if there are exactly $\lambda_1^2$ cells where
  both squares contain ones.  A set of binary \MOFS\ is a set of
  binary frequency squares in which each pair is orthogonal. A set of
  binary \MOFS\ of type $(n;\lambda_0,\lambda_1)$ is \emph{type
    maximal} if there is no square of the type
  $(n;\lambda_0,\lambda_1)$ that is orthogonal to every square in the
  set.

  A \emph{Doehlert-Klee design} consists of points $\cV$ and blocks
  $\cB$, where every pair of points occurs in precisely $\Lambda$
  blocks and every point occurs in precisely $R$ blocks, where
  $R^2=\Lambda|\cB|$.  We show that sets of binary \MOFS\ are
  equivalent to a particular kind of Doehlert-Klee design. In a
  distinct application, Doehlert-Klee designs can also be used to
  construct sets of binary \MOFS\ that are cyclically generated from their
  first rows.  We use these connections to find new constructions for
  sets of type-maximal binary \MOFS.
\end{abstract}

%\textbf{Keywords and MSC Code:}  05B30

\section{Introduction}

A \emph{frequency square} of {\it type}
$(n;\lambda_0,\lambda_1,\dots,\lambda_{m-1})$ is an $n \times n$ array
with entries from the set $\{0,1,\dots,m-1\}$, where entry $i$ occurs
precisely $\lambda_i$ times in each row and in each column, for $0\le
i<m$.  We say that symbol $i$ has \emph{frequency} $\lambda_i$ and
note that $n=\sum_{i}\lambda_i$. Two frequency squares $F_1$ and $F_2$
of type $(n; \lambda_0,\lambda_1,\dots, \lambda_{m_1-1})$ and $(n;
\mu_0,\mu_1,\dots, \mu_{m_2-1})$ respectively, are {\it
  orthogonal} if, when superimposed, each of the $m_1m_2$ possible
ordered pairs $(i,j)$ with $0\leq i<m_1$ and $0\leq j<m_2$, occurs
exactly $\lambda_i\mu_j$ times. A set of frequency squares is said to
be {\it mutually orthogonal} if every two distinct members of the set
are orthogonal. A set of $k$ mutually orthogonal frequency squares
(MOFS) of order $n$ will be referred to as a set of $k$-$\MOFS(n)$, or
simply a set of $k$-MOFS. A set of MOFS is said to have \emph{pure
  type} if we require each of its squares to have the same type,
and otherwise we say it has \emph{mixed type}.  Since the unique
frequency square of type $(n;n)$ is trivially orthogonal to every
other frequency square, we will assume all our frequency squares have
at least two symbols.  We mainly focus on sets of MOFS in which each
of the frequency squares has precisely two symbols and we refer to these as sets of \emph{binary}
MOFS.

A set $\{F_1,F_2,\dots,F_k\}$ of MOFS is said to be \emph{maximal} if
there does not exist a frequency square $F$ that is orthogonal to
$F_i$ for $1 \leq i \leq k$. Note that this definition of maximal does
not require $F$ to be of any particular type. We also define a more
restricted version of maximality.  A set $\{F_1,F_2,\dots,F_k\}$ of
MOFS is said to be \emph{type-maximal} if there does not exist a
frequency square $F$ such that
\begin{itemize}
\item $F$ is orthogonal to $F_i$ for each $1 \leq i \leq k$, and
\item there is some $t \in \{1,2,\dots,k\}$ such that $F_t$ is of the
  same type as $F$.
\end{itemize}
A set of MOFS that is maximal is necessarily type-maximal, but the
converse statement fails \cite{Stinson66}.  To date,
research into the maximality of sets of MOFS has focused primarily on
type-maximality (see \cite{MOFS-EJC,MOFS-2021,JJ-2020}). However,
those papers do not consider sets of MOFS of mixed type, and they use
``maximal'' in place of what we are calling ``type-maximal''.

The study of \MOFS\ includes the study of mutually orthogonal Latin
squares (MOLS) which is one of the most famous topics in
combinatorics.  Many of the applications of MOLS can utilise more
general sets of MOFS. It is also reasonable to hope that developing the theory
of MOFS will lead to greater understanding of MOLS, which remain the
subject of famous open problems such as the prime power conjecture.
Understanding (type-)maximality is very closely related to the key
question of how large a set of MOFS with particular parameters can be
built. In this paper we find several new methods for constructing sets of
type-maximal MOFS.

A \emph{design} is a pair $(\cV,\cB)$ where $\cV$ is a finite set of
\emph{points} and $\cB$ is a collection of \emph{blocks}, each of
which is a subset of $\cV$.  For positive integers $R$ and $\Lambda$,
an $(R,\Lambda)$-design is a design in which every pair of points
occurs in precisely $\Lambda$ blocks and every point occurs in
precisely $R$ blocks.  In general, blocks can be repeated and some
blocks may be empty.  Let $\cD$ be an $(R,\Lambda)$-design $(\cV,\cB)$
and let $V=|\cV|$ and $B=|\cB|$. If $R^2=\Lambda B$, then $\cD$ is a
{\em Doehlert-Klee $(\DK)$ design} or {\em
  \DK$(R,\Lambda)$-design}.  \DK-designs were first introduced in 1972
(see \cite{DK72}), with a sequence of papers written on them within
the same decade by Stanton, Vanstone and McCarthy (see
\cite{MSV76}, \cite{Sta78}, \cite{SV77c}, \cite{SV77b} and
\cite{SV77a}). The purpose of this paper is to explore several
previously unnoticed connections between sets of binary MOFS and \DK-designs.

The structure of this paper is as follows. In Section \ref{s:backg} we
cover background material, including reviewing results from
\cite{Stinson66} that are useful in diagnosing the maximality of
sets of \MOFS. We also present some new results regarding type-maximal and
maximal sets of \MOFS\ that extend upon results and examples from the same
paper. In Section \ref{MOFSareDKs} we show that sets of MOFS are equivalent to
a particular kind of \DK-design, in which the blocks of the
design can be arranged as a square array with elements occurring a
constant number of times in each row and column.  Our main results are
\tref{t:main1} and \tref{dk2mofs}.  In Section \ref{MaxMOFSfromDKs} we
use this equivalence to construct sets of type-maximal \MOFS. Finally, in
Section \ref{cyclic} we use a different approach to construct
sets of type-maximal \MOFS\ from \DK-designs, where the \MOFS\ are
generated cyclically from their first rows.

\section{Background}\label{s:backg}

\subsection{Designs}

A \emph{pairwise balanced design} (PBD) is a design $(\cV,\cB)$ such
that every distinct pair of elements of $\cV$ occurs in precisely
$\Lambda$ blocks, for some positive integer $\Lambda$. Let $V=|\cV|$
and let $\cK=\{|b|:b\in \cB \}$. Then we say a PBD has order $V$ with
block sizes from $\cK$ and index $\Lambda$ and we write
PBD$(V,\cK,\Lambda)$.  Note that blocks of size zero or one have no impact
on the occurrences of pairs in blocks. Therefore in the literature,
blocks of these sizes are usually not considered. However, as will be
become clear in Section~\ref{MOFSareDKs}, it will be convenient for our
purposes to allow $\{0,1\}\in \cK$ in the definition of a PBD. 

If every element occurs in a constant number of blocks, then we say
the PBD is \emph{point-regular}.  In fact, a point-regular PBD, is
equivalent to an $(R,\Lambda)$-design.  Note also that any PBD can be
made point-regular by the addition of blocks of size one.

If we take a PBD 
and impose an additional constraint
that each subset in $\cB$ has constant cardinality $K$ then we have
defined a \emph{balanced incomplete block design} (BIBD). (In fact, any such PBD is also point-regular, so one could also say that a BIBD is an $(R,\Lambda)$-design with $|\cK|=1$.) 

Let $V=|\cV|$ and $B=|\cB|$. We write the parameters of a BIBD as
$(V,B, R,K,\Lambda)$.  If $K\geq 2$, $V$, $K$ and $\Lambda$ determine
the remaining two parameters via $R=\Lambda(V-1)/(K-1)$ and $B=VR/K$,
hence we often refer to a BIBD as a $(V,K,\Lambda)$-design.

\subsection{Mutually orthogonal frequency squares}

For the majority of this paper we require our sets of MOFS to have pure type, 
and in this case we refer to a set of $k$-MOFS in which each frequency square 
is of type $(n;\lambda_0,\lambda_1,\dots, \lambda_{m-1})$ as 
a set of $k$-$\MOFS(n; \lambda_0,\lambda_1,\dots, \lambda_{m-1})$. 
Note that applying any
surjection onto $\{0,1\}$ (or, indeed, any other function) to the
symbols of a frequency square does not alter orthogonality between
that square and any other frequency square. For this reason, a set of
MOFS is maximal if and only if it cannot be extended by a binary
frequency square.

For the remainder of this section we consider methods for diagnosing
maximality of sets of MOFS. Such techniques often consider the $\Z_w$-sum of
a set of MOFS, which is the matrix obtained from the set using normal matrix
addition, mod some positive integer $w$.  The following theorem was
proved in \cite{Stinson66}. Let $J_{a,b}$ denote an $a \times b$ array
with every entry equal to one. For simplicity, we will use $J_{a}$ to
denote $J_{a,a}$ and we will use $0$ to denote a block of zeros whose dimensions are implied.

\begin{theorem}\label{thm-block-structure}
  Let $w$ be a non-negative integer.  Let $\cF=\{F_1,F_2,\dots,F_k\}$
  be a set of binary \MOFS\ where $F_t$ has type
  $(n;\lambda_{0,t},\lambda_{1,t})$ for $1 \leq t \leq k$. Suppose
  $\cF$ has $\Z_w$-sum that, up to permutation of the rows and
  columns, has the following structure of constant blocks
  \begin{equation}\label{block-matrixJ}
    \left[ \begin{array}{cc}
	{x_1}J_{a,b} & {x_2}J_{a,n-b} \\
	{x_3}J_{n-a,b} & {x_4}J_{n-a,n-b}
      \end{array}\right],
  \end{equation}
  where $x_1,x_2,x_3,x_4$ are non-negative integers satisfying
  $x_1+x_4\equiv x_2+x_3 \bmod w$. Let $F$ be a square of type
  $(n;\mu_0,\mu_1,\dots,\mu_{m-1})$. If $\cF \cup \{F\}$ is a set of
  $(k+1)$-\MOFS\ then for $0\leq i<m$ we have
  \begin{equation}\label{(1,1)count}
    \mu_i \sum_{t=1}^k \lambda_{1,t}\equiv a\mu_i(x_2-x_4)+b\mu_i(x_3-x_4) + \mu_i n x_4 \bmod w.
  \end{equation}
\end{theorem}

\begin{corollary}\label{cor-block-structure}
  Let $\cF=\{F_1,F_2,\dots,F_k\}$ be a set of binary \MOFS\ where
  $F_t$ has type $(n; \lambda_{0},\lambda_{1})$ for $1 \leq t \leq k$
  and suppose some non-negative integers $x_1,x_2,x_3,x_4$ satisfy the
  hypotheses of \tref{thm-block-structure}. Then, if either $\lambda_0$
  or $\lambda_1$ is
  coprime to $w$ and
  \begin{equation}\label{cor-(1,1)count}
    k\lambda_1\not\equiv a(x_2-x_4)+b(x_3-x_4) + n x_4 \bmod w,
  \end{equation}	
  then the set $\cF$ is type-maximal. 	
\end{corollary}

In \cite{Stinson66} it was not made clear whether in
\tref{thm-block-structure} we should allow the degenerate case where
$a\in\{0,n\}$ or $b\in\{0,n\}$. We now show that it is a moot point in
the sense that \eref{(1,1)count} is always satisfied in this
degenerate case. It suffices to consider the case $a=0$. Counting the
ones in the first rows of $F_1,\dots,F_t$ we find that
  \begin{equation}
    \sum_{t=1}^k \lambda_{1,t}\equiv bx_3+(n-b)x_4
    \equiv a(x_2-x_4)+b(x_3-x_4) + n x_4 
    \bmod w,
  \end{equation}
from which \eref{(1,1)count} follows.

The extreme case of degeneracy in \eref{block-matrixJ} is where the
set of \MOFS\ has constant $\Z_w$-sum (so $a\in\{0,n\}$ and
$b\in\{0,n\}$). Such sets of \MOFS\ do exist, although we show in
\tref{t:constZ} that no set of binary \MOFS\ of pure type has constant
$\Z$-sum. At the end of Section \ref{MOFSareDKs}, we will see examples of
sets of \MOFS\ with constant $\Z_w$-sum.  We will then see in
Section~\ref{MaxMOFSfromDKs} that in some circumstances such
sets of \MOFS\ can be used to construct maximal sets of \MOFS, notwithstanding
our comments above that \tref{thm-block-structure} cannot be applied
directly to show maximality.

A set of binary $k$-$\MOFS(n;n-1,1)$ is equivalent to an {\it equidistant
permutation array} (EPA). An EPA $A(n,d;k)$ is a $k \times n$ array
in which each row contains each integer from $1$ to $n$ precisely
once, and every pair of rows differ in precisely $d$ positions. Now,
each square in a set of $k$-$\MOFS(n;n-1,1)$ is a permutation matrix. Taking
the corresponding permutations as the rows of a $k \times n$ array
yields an $A(n,n-1;k)$.

An easy way to use sets of MOFS to obtain sets of MOFS of larger order is to use
\emph{dilation}. To dilate a set of $k$-$\MOFS(n;n-\lambda,\lambda)$
we replace each zero by a block of zeros of order $\lambda$, and
replace each one by a block of ones of order $\lambda$. This yields
a set of $k$-\MOFS$(\lambda n;\lambda n-\lambda^2,\lambda^2)$.

Next we give a simple generalisation of Example 5.6 from \cite{Stinson66}.
Our example involves dilating a set of MOFS that is equivalent to an
EPA.

\begin{example}\label{eg:dilperm}
Let $n$ be an integer, and let $d>1$ be any divisor
of $n$. Take the $d-1$ permutation matrices corresponding to the
powers of the cycle $(23\cdots d)$ and then dilate by a factor of
$\lambda=n/d$.  This yields a set of binary
$(d-1)$-\MOFS$(n;n-\lambda,\lambda)$. Their $\Z$-sum has the block
structure \eref{block-matrixJ} with $x_1=d-1$, $x_2=x_3=0$ and
$x_4=1$. So if we take $q$ to be any divisor of $d$ then
$x_1+x_4\equiv x_2+x_3 \bmod q$.
\end{example}

Suppose we have a frequency square of type $(n;n-\mu,\mu)$ which
extends our set of MOFS in \egref{eg:dilperm}. If $\gcd(\mu,q)=1$ then
Corollary 5.4 from \cite{Stinson66} says $(d-1)\lambda=-2\lambda \bmod
q$ so $q\mid(d+1)\lambda=n+\lambda$. But $q\mid d$ and $d\mid n$ so
this last condition is equivalent to $q\mid\lambda$.

Hence we arrive at the following result.

\begin{theorem}\label{type-max}
The set of binary $(d-1)$-\MOFS$(n;n-\lambda,\lambda)$ constructed in
\egref{eg:dilperm} are type-maximal unless $\lambda$ is divisible by
every prime divisor of $d$.
\end{theorem}

\begin{proof}
If there is any $p\mid d$ that is not a factor of $\lambda$ then take
$q=p$ and we have $\gcd(\lambda,q)=\gcd(\mu,q)=1$.
\end{proof}
 
In various cases such as when $n$ is square-free, or $\lambda=1$, for
example, we can always find the required $p$.  In fact combining these two
cases we can show our example is maximal, not just type-maximal.

\begin{theorem}\label{t:max}
The set of binary $(d-1)$-\MOFS$(n;n-\lambda,\lambda)$ constructed in
\egref{eg:dilperm} is maximal when $\lambda=1$ and $n$ is square-free.
\end{theorem}

\begin{proof}
Suppose we can extend our set of MOFS with a square of type
$(n;n-\mu,\mu)$. We can choose $q$ to be a prime that divides $n$ but
not $\mu$ (such exists, since $\mu<n$ and $n$ is square-free).  Then
$\gcd(\mu,q)=1$ and $q$ does not divide $\lambda$.
\end{proof}

It is worth noting that we cannot generalise \tref{t:max} to prove
maximality in other cases.  If $s^2\mid n$ then taking $\mu=n/s$ makes
it impossible to find $q\mid n$ such that $\gcd(q,\mu)=1$.  If
$\lambda>1$ then taking $\mu=d$ means that any $q\mid d$ does not
satisfy $\gcd(\mu,q)=1$.  (NB: If $\lambda=1$ then $\mu=d$ is not an
allowable choice, but otherwise it is).

\section{Binary MOFS are Doehlert-Klee designs}\label{MOFSareDKs}

In this section we show that a set of binary MOFS of pure type
is equivalent to a particular kind of DK-design. We will need the
following definitions.  Let $\cP_1$ and $\cP_2$ be partitions of the
same set.  We say that the pair of partitions $(\cP_1,\cP_2)$ is {\em
  orthogonal} if $|P_1\cap P_2|\le1$ for each $P_1\in \cP_1$ and $P_2\in
\cP_2$.  Next, we define a partition of a set of sets (in our context,
a set of blocks) to be an \emph{equipartition} if each point occurs
the same number of times among the blocks in each part. The case when
each point occurs once per part is known in design theory as a
\emph{resolution} and if a design has orthogonal resolutions it is
said to be \emph{doubly resolvable}. The situation we are studying is
a direct generalisation of these concepts.

Starting with a set of binary MOFS, we can produce a collection of blocks
as follows.  Let $\cF=\{F_1,F_2,\dots,F_k\}$ be a set of binary
$k$-$\MOFS(n)$.  For $1\leq i\leq n$ and $1\leq j\leq n$, let $B_{i,j}\subseteq
\{1,2,\dots,k\}$ be the set of points $x$ such that cell $(i,j)$ of
$F_x$ contains the entry one.
We say that these $n^2$ blocks are {\em derived} from $\cF$. For
example, below is a set of $2$-$\MOFS(3;2,1)$, along with the set of nine
blocks derived from it (presented as an array).
\begin{equation}\label{e:DKeg}
\left\{\left[\begin{array}{cccccc}
1&0&0\\
0&1&0\\
0&0&1\end{array} \right]
,
\left[ \begin{array}{cccccc}
1&0&0\\
0&0&1\\
0&1&0\end{array} \right]\right\}
\hspace{1.2cm}
\left[\begin{array}{cccccc}
\{1,2\}&\emptyset&\emptyset\\
\emptyset&\{1\}&\{2\}\\
\emptyset&\{2\}&\{1\}\end{array}\right].
\end{equation}
The following lemma describes the properties of the set of blocks derived 
from a set of \MOFS.

\begin{lemma}
  Let $\cF$ be a set of binary $k$-\MOFS$(n;\lambda_0,\lambda_1)$.
  Suppose that $\cB$ is the set of blocks derived from $\cF$.
  Then we have
  \begin{enumerate}
  \item $|\cB|=n^2$; 
  \item each block in 
    $\cB$ is a subset of $\cV:=\{1,2,\dots,k\}$;
  \item each element of $\cV$ occurs in $\lambda_1n$ blocks;  	 
  \item each pair of elements from $\cV$ occurs in precisely 
    $\lambda_1^2$ blocks. 
  \end{enumerate}
\end{lemma} 

It follows that the set of blocks derived from such an $\cF$ is a
\DK$(\lambda_1n,\lambda_1^2)$-design. This \DK-design has the
additional property that its $n^2$ blocks can be arranged in an array,
where each element occurs $\lambda_1$ times in each row and column.
From this array, let $\cP_1$ and $\cP_2$ be the partitions of $\cB$
corresponding to the rows and columns.  Then, $\cP_1$ and $\cP_2$ are
each equipartitions of type $n^n$ (meaning they consist of $n$ parts
of cardinality $n$).  Moreover, the pair of partitions $\cP_1$ and
$\cP_2$ is {\em orthogonal}.

The partitions corresponding to the example \eref{e:DKeg}
are shown below.
\begin{align*}
\cP_1&=\{\{B_{1,1},B_{1,2},B_{1,3}\},\{B_{2,1},B_{2,2},B_{2,3}\},\{B_{3,1},B_{3,2},B_{3,3}\}\},\\
\cP_2&=\{\{B_{1,1},B_{2,1},B_{3,1}\},\{B_{1,2},B_{2,2},B_{3,2}\},\{B_{1,3},B_{2,3},B_{3,3}\}\},\\
B_{1,1}&=\{1,2\},B_{1,2}=B_{1,3}=B_{2,1}=B_{3,1}=\emptyset,
B_{2,2}=B_{3,3}=\{1\},B_{2,3}=B_{3,2}=\{2\}.
\end{align*}
Note that some blocks derived from \eref{e:DKeg} are repeated and some
blocks are empty. 
The former property is not uncommon in design theory,
but it is rarer to allow empty blocks.

We have shown the following.

\begin{theorem}\label{t:main1}
  Suppose a set of blocks $\cB$ is derived from a set of binary
  $k$-$\MOFS(n;\lambda_0,\lambda_1)$.  Then $\cB$ is a
  \DK$(\lambda_1n,\lambda_1^2)$-design with $k$ elements and $n^2$
  blocks. Furthermore, there are two orthogonal equipartitions,
  ${\cP_1}$ and ${\cP_2}$ of $\cB$, 
  each of type $n^n$. 
\end{theorem} 

In fact, the process above is reversible.  Let $\lambda_1$ and $n$ be
positive integers. Suppose there exists a
\DK$(\lambda_1n,\lambda_1^2)$-design $(\cV,\cB)$ with $|\cB|=n^2$,
together with two orthogonal equipartitions, ${\cP_1}$ and ${\cP_2}$
of $\cB$, each of type $n^n$.  By the definition of equipartition,
each element of $\cV$ occurs $\lambda_1$ times in each part of
${\cP_1}$ and ${\cP_2}$.

Let $k=|\cV|$ and let ${\cP_1}=\{R_1,R_2,\dots,R_n\}$ and
${\cP_2}=\{C_1,C_2,\dots,C_n\}$.  For $1\le i,j\le n$ 
orthogonality implies that $|R_i\cap C_j|\le1$, but then a counting
argument shows that in fact equality holds.  Define $B_{i,j}$ to be
the unique element of $R_i\cap C_j$.  Next, construct a set
$\cF=\{F_1,F_2,\dots,F_k\}$ of $n\times n$ binary arrays where cell
$(i,j)$ of $F_x$ contains the entry one if and only if $x\in B_{i,j}$, for $1\le
x\le k$ and $1\le i,j\le n$. We say that the set $\cF$ is {\em
  derived} from the \DK-design $(\cV,\cB)$.

\begin{theorem}\label{dk2mofs}
  Suppose there exists a \DK-design $(\cV,\cB)$
  with $|\cV|=k$ and $|\cB|=n^2$ blocks together with two orthogonal
  equipartitions, ${\cP_1}$ and ${\cP_2}$ of $\cB$, each of type $n^n$.  
   Let $\cF$ be the set derived from $(\cV,\cB)$.
  Then $\cF$ is a set of $k$-\MOFS\-$(n;n-\lambda_1,\lambda_1)$.
\end{theorem}

Consider a set of $k$-MOFS with constant sum in $\Z$. The blocks derived from this set of MOFS will have constant size. Furthermore, if a \DK-design is derived from a set of MOFS, and has constant block size, then it must be trivial (that is, $\lambda_1\in \{0,n\}$).  

\begin{lemma}
  Let $\cD$ be the \DK-design derived from a set of \MOFS$(n;\lambda_0,\lambda_1)$. 
  If $\cD$ has constant block size, then $\lambda_1\in \{0,n\}$ and thus 
  $\cD$ is a $(0,0)$- or $(n^2,n^2)$-design and is therefore trivial.
\end{lemma}

\begin{proof}
Suppose the collection of blocks $\cB$ derived from a set of
$k$-$\MOFS(n;\lambda_0,\lambda_1)$ is a BIBD with parameters
$(V,B,R,K,\Lambda)$.  Then, $V=k$, $\Lambda=\lambda_1^2$,
$R=n\lambda_1$ and $B=n^2$. Since a BIBD satisfies $VB=RK$ we have
that $K=kn/\lambda_1$ (if $\lambda_1\neq 0$).  Furthermore, since
$\Lambda(V-1)=R(K-1)$, it follows that
$$\lambda_1^2(k-1)= n(nk-\lambda_1).$$
Solving this quadratic for $n>0$ we find that $n=\lambda_1$ so $V=K$.   
\end{proof}

We have proven the following theorem. 
 
\begin{theorem}\label{t:constZ}
  Suppose that $\{\lambda_0,\lambda_1\}\cap\{0,n\}=\emptyset$.
  Then there does not exist a set of \MOFS\ of 
  type $(n;\lambda_0,\lambda_1)$ with constant sum in $\Z$. 
\end{theorem}

On the basis of \tref{t:constZ}, it is fair to say that there are no
interesting examples of sets of binary \MOFS\ of pure type with constant
$\Z$-sum. Luckily for us, there are such examples with constant
$\Z_w$-sum.  In the next section we will see how some of them can be
used to build maximal sets of \MOFS\ (cf.~our comments after
\Cref{cor-block-structure} about that result not applying). So for the
remainder of this section we provide some examples of sets of \MOFS\ with
constant $\Z_w$-sum.

Example 3.2 in \cite{Stinson66} displays a set of
$4$-$\MOFS(8;4,4)$ which has constant $\Z_2$-sum. It is not
type-maximal, because each of its members is orthogonal to
\[
\left[
\begin{array}{cc}
0&J_4\\
J_4&0\\
\end{array}
\right].
\]
We next give some examples of sets of MOFS with constant $\Z_3$-sum. In
each case the sets of MOFS are given in superimposed form. We start with
a set of $6$-$\MOFS(5;3,2)$ which is maximal %[[NOT JUST type-max]]
\[
\left[\begin{array}{ccccc}
000000&000111&001011&110100&111000\\
010101&000000&011100&101010&100011\\
101001&101100&000000&010011&010110\\
100110&011010&110001&000000&001101\\
011010&110001&100110&001101&000000\\
  \end{array}\right]
\]
and a set of $8$-$\MOFS(6;3,3)$, which is also maximal,
\[
\left[\begin{array}{cccccc}
00000011&00001100&00111101&11110100&11010011&11101010\\
00010100&01110011&01001000&10101011&11100101&10011110\\
01001111&01100000&10110110&10010000&00111011&11001101\\
10111001&10101110&10000001&01000010&01011110&01110101\\
11111000&11011001&11000111&00101111&00100100&00010010\\
11100110&10010111&01111010&01011101&10001000&00100001\\
  \end{array}\right].
\]
However, not all sets of MOFS with constant $\Z_3$-sum are
maximal. Here is a set of $7$-$\MOFS(6;3,3)$
\[
\left[\begin{array}{cccccc}
1111111&0000001&0000010&0011110&1101100&1110001\\
0000100&0010111&0101101&1011001&1101010&1110010\\
0100000&0111001&1010110&0101011&1000111&1011100\\
0111010&1001110&1011001&1000000&0110101&0100111\\
1010011&1101010&0111100&1100101&0010000&0001111\\
1001101&1110100&1100011&0110110&0011011&0001000\\
  \end{array}\right]
\]
which is not even type-maximal, because it
can be extended by
\[\left[\begin{array}{cccccc}
    0&0&0&1&1&1\\
    0&1&1&0&0&1\\
    1&1&1&0&0&0\\
    1&1&0&1&0&0\\
    0&0&0&1&1&1\\
    1&0&1&0&1&0\\
\end{array}\right].
\]

%[[Can we find an example which is type-max but not max??]]

\section{Maximal sets of MOFS from Doehlert-Klee designs}\label{MaxMOFSfromDKs}

Our strategy in this section will be to use \DK-designs with block
sizes constant modulo $w$. These can be used to create sets of type-maximal
binary MOFS in conjunction with \Cref{cor-block-structure}, thanks to
the following lemma.

\begin{lemma}\label{addingones}
  Let $\cF=\{F_1,F_2,\dots,F_k\}$ be a set of binary
  \MOFS-$(n;\lambda_{0},\lambda_{1})$ with $\Z$-sum $M$.  Then there
  exists a set $\cF'=\{F_1',F_2',\dots,F_k'\}$ of binary
  \MOFS-$(n+k\lambda_1;n+(k-1)\lambda_1,\lambda_1)$ with $\Z$-sum
   \begin{equation}\label{block-matrixMOOJ}
     \left[ \begin{array}{cc}
 	 M & 0 \\
 	 0 & J_{k\lambda_1}
       \end{array}\right].
   \end{equation}
\end{lemma}

 \begin{proof}
It is a simple matter to decompose the complete bipartite graph
$K_{k\lambda_1,k\lambda_1}$ into $k$ spanning $\lambda_1$-regular
subgraphs. The biadjacency matrices of such graphs give us frequency
squares $F_1''$, $F_2'',\dots,F_k''$ of type
$(k\lambda_1;(k-1)\lambda_1,\lambda_1)$ whose $\Z$-sum is
$J_{k\lambda_1}$.  Then for each $i\in \{1,2,\dots,k\}$ we
define $F_{i}'$ to be
\begin{equation*}
  \left[ \begin{array}{cc}
      F_i & 0 \\
      0 & F_{i}''
    \end{array}\right].\qedhere
\end{equation*}
\end{proof}

The next lemma describes some conditions on \DK-designs
which are sufficient to construct sets of type-maximal MOFS.

\begin{lemma}\label{dk2typemax} 
  Suppose there exists a \DK$(\lambda_1n,\lambda_1^2)$-design $\cD$ on
  $k$ points with block size set $\cK$ such that the blocks have
  orthogonal equipartitions of type $n^n$ where each point occurs
  $\lambda_1$ times in each part.  Suppose furthermore that there is
  some constant $w$ such that
  \begin{enumerate}
  \item
    each element of $\cK$ is congruent to $w-1$ modulo $w$;
  \item either $\lambda_1$ or $n+(k-1)\lambda_1$ is coprime to $w$;
  \item $n\not\equiv0\bmod w$.
  \end{enumerate} 
  Then there exists a type-maximal set $\cF$ of binary
  $k$-\MOFS$(n+k\lambda_1;n+(k-1)\lambda_1,\lambda_1)$.
\end{lemma}

\begin{proof}
Let $\cF=\{F_1,F_2,\dots,F_k\}$ be the set derived from $\cD$.  From
\tref{dk2mofs}, $\cF$ is a set of binary
$k$-$\MOFS(n;n-\lambda_1,\lambda_1)$.  From the given properties of
$\cK$, the $\Z_w$-sum of $\cF$ is equal to $(w-1)J_{n}$. Let
$n'=n+k\lambda_1$ and $\lambda_0=n'-\lambda_1$.  From
\lref{addingones}, there exists a set $\cF'$ of binary
$k$-\MOFS$(n';\lambda_0,\lambda_1)$ whose $\Z_w$-sum is equal to
\begin{equation*}
  \left[ \begin{array}{cc}
      (w-1)J_{n} & 0 \\
      0 & J_{k\lambda_1}
    \end{array}\right].
\end{equation*} 
By Condition {\em 2}, at least one of $\lambda_0$, $\lambda_1$ is
coprime to $w$.  With the aim to apply \Cref{cor-block-structure}, set
$x_1=w-1$, $x_2=x_3=0$, $x_4=1$ and $a=b=n$. Then we have
\[
a(x_2-x_4)+b(x_3-x_4)+n'x_4-k\lambda_1
=-2n+(n+k\lambda_1)-k\lambda_1=-n.
\] 
Since $n\not\equiv0\bmod w$, Equation \eref{cor-(1,1)count} is
satisfied.  Thus $\cF$ is type-maximal.
\end{proof}

Next we give an example of an application of \lref{dk2typemax}
with $w=3$, $n=13$, $\lambda_1=4$ and $k=8$.
We construct a \DK$(52,16)$-design on point set $\Z_8$ 
with all block sizes congruent to $2 \bmod 3$. 
We find orthogonal equipartitions of its 169 blocks as follows.
Divide the quadrants of a $13\times 13$ array into 
$Q_1$ (a $5\times 5$ array),
$Q_2$ (a $5\times 8$ array),
$Q_3$ (an $8\times 5$ array) 
and
$Q_4$ (an $8\times 8$ array).
In $Q_1$ we place the following blocks, each $5$ times, once per row
and once per column:
$$\{0,1,2,3,4,5,6,7\},\{0,4\},\{1,5\},\{2,6\},\{3,7\}.$$
In $Q_2$, let the first column be given by the blocks: 
$$\{4,7\},\{0,5\},\{3,5\},\{1,2\},\{6,7\}.$$
Then increment these modulo $8$ to obtain the remaining columns of $Q_2$. 
In $Q_3$, let the first row be given by the blocks: 
$$\{2,5\},\{4,6\},\{0,3\},\{2,7\},\{1,7\}.$$
Then increment these modulo $8$ to obtain the remaining rows of $Q_3$. 
In $Q_4$, let the first row be given by the blocks: 
$$\{1,2,3,4,5\},\{5,6,7,0,1\},\{0,4\},\{4,6\},\{2,3\},\{5,7\},\{1,6\},\{0,3\}.$$
Then increment these modulo $8$ along diagonals to obtain the
remaining rows of $Q_4$. It is routine to check that we have constructed the
claimed \DK$(52,16)$-design.
Thus, by \lref{dk2typemax}, there exists a type-maximal set of
$8$-$\MOFS(45;41,4)$.

Now, if there exists a set of $k$-$\MOFS(n;\lambda_0,\lambda_1)$ then
for any positive integer $q$, there also exists a set of
$k$-$\MOFS(nq;\lambda_0q,\lambda_1q)$, which can be obtained by
dilations. In particular, by dilating the set of $8$-$\MOFS(13;9,4)$ derived
from the $13\times13$ array above, we can obtain a set of
$8$-$\MOFS(13q;9q,4q)$, for any integer $q$.  In turn, using
\lref{dk2typemax}, for any $q$ not divisible by $3$
there exists a type-maximal set of $8$-$\MOFS(45q;41q,4q)$.
This is despite the fact that the original set of $8$-$\MOFS(13;9,4)$
are not type-maximal, as evidenced by the following frequency square,
which can be used to extend it to a set of $9$-$\MOFS(13;9,4)$:
\[\left[\begin{array}{ccccccccccccc}
0&1&1&1&1&0&0&0&0&0&0&0&0\\
0&1&1&1&1&0&0&0&0&0&0&0&0\\
0&1&1&1&1&0&0&0&0&0&0&0&0\\
0&0&0&1&1&0&0&0&0&0&1&0&1\\
1&0&0&0&0&0&0&0&0&1&0&1&1\\
1&0&0&0&0&0&0&0&0&1&1&1&0\\
1&0&1&0&0&0&0&0&0&0&0&1&1\\
0&0&0&0&0&0&0&0&0&1&1&1&1\\
1&1&0&0&0&0&0&0&0&1&1&0&0\\
0&0&0&0&0&1&1&1&1&0&0&0&0\\
0&0&0&0&0&1&1&1&1&0&0&0&0\\
0&0&0&0&0&1&1&1&1&0&0&0&0\\
0&0&0&0&0&1&1&1&1&0&0&0&0\\
  \end{array}\right]
\]
We also note that when $q$ is a large power of $2$ our
set of $8$-$\MOFS(45q;41q,4q)$ are much smaller type-maximal sets than those
given by Corollary 9 from \cite{MOFS-2021}.

%[[Can we find a FS of type (45,42,3) orthogonal to the above, to show that
%it is not maximal even though it is type-maximal?]]

\section{An approach using cyclic frequency squares}\label{cyclic}

In the previous section we were able to find two equipartitions of a
\DK-design with the required properties in the case when
most of the blocks were of size $2$; in other cases finding
such equipartitions may be harder. In this section, we instead
consider when sets of MOFS are cyclically generated, using \DK-designs
to determine the first row only.

\begin{lemma}\label{lem-DK}
Suppose there exists a \DK$(R,\Lambda)$-design $\cD$ on $V$ points,
with B blocks and with block size set $\cK$ such that each element of
$\cK$ is congruent to $x$ modulo $w$ for some constants $x$ and $w$.
Then there exists a set $\cF$ of $V$-\MOFS$(B+VR;B+VR-R,R)$ whose
$\Z_w$-sum is equal to
\begin{equation}\label{block-matrixx00J}
\left[ \begin{array}{cc}
xJ_{B} & 0 \\
0 & J_{VR}
\end{array}\right].  
\end{equation}
\end{lemma}

\begin{proof}
Let $I$ be the $V\times B$ incidence matrix for $\cD$.  By
definition, the sum of each row of $I$ is $R$ and the sum of each
column is congruent to $x$ modulo $w$.  Furthermore, the dot product of
any two rows of $I$ is $\Lambda$. 
 
Let $F_i$ be the frequency square of type $(B;B-R,R)$ whose first row
is the $i$-th row of $I$ and each subsequent row is formed by
cyclically shifting the previous row one place to the right.
Superimposing $F_i$ and $F_j$, where $i\neq j$, results in $B\Lambda$
occurrences of the pair $(1,1)$.  By the definition of a \DK-design,
$B\Lambda = R^2$.  It follows that $\{F_1,F_2,\dots,F_V\}$
is a set of binary MOFS of type $(B;B-R,R)$, with $\Z_w$-sum given
by $xJ_{B}$.

The result then follows from \lref{addingones}.
\end{proof}

Using \lref{lem-DK} and \tref{thm-block-structure}, we now have another
method of creating type-maximal sets of MOFS.

\begin{theorem}\label{t:puttingittogether}
  Suppose there exists a \DK$(R,\Lambda)$-design on
  $V$ points with $B$ blocks and block size set $\cK$, such that 
  each element of $\cK$ is congruent to $w-1\bmod w$ for some
  constant $w$.  Furthermore, suppose that at least one of $RB$ or
  $B^2$ is not divisible by $w$.  Then there exists a type-maximal set
  of $V$-\MOFS$(B+VR;B+VR-R,R)$.
\end{theorem}

\begin{proof}
  Aiming for a contradiction, suppose that the set of $V$-MOFS $\cF$ in
  \eref{block-matrixx00J} are not type-maximal. So suppose there
  exists a frequency square $F$ of type $(n;\lambda_0,\lambda_1)$ such
  that $\cF \cup \{F\}$ is a set of $(V+1)$-MOFS.  We apply
  \tref{thm-block-structure} with $x_1=w-1$, $x_2=x_3=0$, $x_4=1$,
  $a=b=B$, $k=V$, $n=B+VR$, $\mu_1=\lambda_1=R$ and
  $\mu_0=\lambda_0=n-R$.  Then \eref{(1,1)count} implies that:
  \[
  \mu_iRV\equiv -2\mu_iB+ \mu_i(B+VR) \bmod{w}, 
  \]
  which simplifies to $\mu_iB\equiv 0\bmod{w}$.
  Since $\mu_1=R$ it follows that $RB\equiv 0\bmod{w}$.
  However, we also have $\mu_0=B+VR-R$, which means that
  \[
  0\equiv(B+VR-R)B=B^2+(V-1)RB\equiv B^2\bmod w.
  \]
  The result follows.   
\end{proof}

It will never make sense to apply \tref{t:puttingittogether} to a
design with empty blocks, as this forces $w=1$.  This observation
rules out many classical constructions for \DK-designs (see, for
example, \cite{MSV76} and \cite{SV77c}) as candidates for yielding maximal sets of MOFS via
\tref{t:puttingittogether}.

However, for $w=2$ we can give a concrete example.  We begin with a
cyclic PBD on $27$ points with starter blocks in $\Z_{27}$ of
$\{0,1,6,13,24\}$ and $\{0,2,10\}$.  By adding all the 27 singletons
from $\Z_{27}$, we obtain a \DK-design $(\cV,\cB)$ with $V=27$,
$B=81$, $\Lambda=1$ and $R=9$.  Setting $w=2$, from
\tref{t:puttingittogether} there exists a type-maximal set of
$27$-\MOFS$(324;315,9)$.

Note that \tref{type-max} gives us a type-maximal set of
$35$-\MOFS$(324;315,9)$, using $n=324$, $d=36$, and $\lambda=9$.

\subsection{Large examples}\label{largeexample} 

For $w>2$ (specifically $w=8$), we have been able to find large
examples satisfying \tref{t:puttingittogether}, which we exhibit in
this subsection. Specifically, we show the following.

\begin{lemma}
Let $V=12615$ and for sufficiently large odd $g$, let $z=15g^2V$.  Let
$R=3074z$ and $B=4724738z$.  Then there exists a type-maximal
set of $V$-\MOFS$(B+VR;B+VR-R,R)$.
\end{lemma}

Let $\Lambda=2z$, $w=8$, $B_{15}=713168z$ and $B_7=4011570z$ and observe
that: $B=B_{15}+B_7$, $R^2=\Lambda B$ and $RB \not\equiv 0\bmod w$.
From \tref{t:puttingittogether}, it suffices to show that there exists
a \DK$(R,\Lambda)$-design on $V$ points with $B$ blocks and block size
set $\{7,15\}$, where there are $B_i$ blocks of size $i$ for each
$i\in \{7,15\}$.
  
Such a \DK-design is in fact also a point-regular
PBD$(V,\{7,15\},\Lambda)$, so we require that $VR=7B_7+15B_{15}$ and
$\Lambda V(V-1)= B_{15}\times 15\times 14 + B_7\times 7\times 6$; it
may be observed that these all hold.  To construct such a PBD, we make
use of Corollary 1.4 of \cite{Mon19}, which states:

\begin{theorem}\label{t:Mon} 
  Let $F$ be a graph with chromatic number $\chi(F) \geq 4$ and let
  $\epsilon>0$. Any sufficiently large $F$-divisible graph $G$
  with minimum degree $\delta(G) \geq (1-1/(100\chi(F)) + \epsilon)|G|$ has an
  $F$-decomposition.
\end{theorem}

\begin{lemma}
For some sufficiently large $g$, let $z=15g^2V$.  Let $V,$ $\Lambda$,
$B_7$ and $B_{15}$ be defined as above.  Then there exist a
point-regular \PBD$(V,\{7,15\},2\Lambda)$ with $B_7$ blocks of
size $7$ and $B_{15}$ blocks of size $15$.
\end{lemma}

\begin{proof}
Let $F$ be the complete graph on $15$ vertices and let $G$ be the
complete multipartite graph with $V$ parts, each of size $g$. Since
each vertex of $G$ has degree $(V-1)g$ which is divisible by $14$, the
degree of $F$ divides the degree of $G$.  Also $G$ has $g^2V(V-1)/2$
edges which is divisible by $105$ (the number of edges of $F$).  Thus
$G$ is $F$-divisible.  The chromatic number of $F$ is $15$. Then,
$\delta(G)/|G|=(V-1)/V> 1-1/1500$.  Thus by \tref{t:Mon}, for
sufficiently large $g$, there exists a decomposition of $G$ into $F$.

In such a decomposition, each vertex belongs to $g(V-1)/14$ copies of $K_{15}$. There are 
$V(V-1)g^2/210$ copies of $K_{15}$ altogether. 
Collapsing each partite set to a single vertex, and taking $30$ copies
of each block, we obtain a $(V,15,15g^2(V-1)/7)$-design $\cD$ with
$g^2V(V-1)/7=1802g^2V$ blocks of size $15$.

Next, by \cite[\S II.3.1]{HCD} (or by taking all the Fano planes
within the projective space on 15 points), there exists a
$(15,7,3)$-design.  Thus we can replace any set of three of the copies
of a block of size $15$ with $15$ blocks of size $7$.  In
particular, we replace $B_7/(15V)=4011570g^2$ sets of $3$ blocks of
size $15$ in $\cD$ to obtain $B_7/V$ blocks of size $7$.  Observe
there are $B_{15}/V$ blocks of size $15$ left.

Finally, to create a PBD which is point-regular, we assume the points
of $V$ are labelled by ${\mathbb Z}_{15}$ and take the cyclic shift of
each block under the additive action of this group on~$V$.
\end{proof}

\section{Discussion}

In considering sets of $\MOFS(n)$ one of the main research challenges has been
trying to establish what cardinalities maximal sets can have.  There
are few known constructions in the case when $n$ is odd. Our newfound
connections between \DK-designs and sets of  \MOFS\ provide general methods for
constructing maximal sets of \MOFS, including when $n$ is odd. We also
now have motivation to further develop the theory of \DK-designs.

In \tref{t:constZ} we showed that there are no interesting examples
of sets of \MOFS\ of pure type with constant $\Z$-sum. It remains open whether
there are examples of mixed type with constant $\Z$-sum.

  \let\oldthebibliography=\thebibliography
  \let\endoldthebibliography=\endthebibliography
  \renewenvironment{thebibliography}[1]{%
    \begin{oldthebibliography}{#1}%
      \setlength{\parskip}{0.4ex plus 0.1ex minus 0.1ex}%
      \setlength{\itemsep}{0.4ex plus 0.1ex minus 0.1ex}%
  }%
  {%
    \end{oldthebibliography}%
  }

\end{document}